%% file: ex_article.tex
\documentclass[onefignum,onetabnum]{siamart171218}
\usepackage{color}

\newcommand{\cX}{\mathcal{X}}
\newcommand{\cZ}{\mathcal{Z}}

\input{ex_shared}

\ifpdf
\hypersetup{
   pdftitle={FINITE-TIME LINEAR-QUADRATIC OPTIMAL CONTROL OF PDAEs},
  pdfauthor={A. Alalabi and K. Morris}
}
\fi

\begin{document}

\maketitle

\begin{abstract}
We consider finite-time linear-quadratic control of partial differential-algebraic equations (PDAEs) . The discussion is restricted to those that are radial with index $0$; this corresponds to a nilpotency degree of 1. We  establish the existence of a unique minimizing optimal control.  A projection is used to derive a system of  differential Riccati-like equation coupled with an algebraic equation, yielding the solution of the optimization problem in a feedback form. This generalizes the well-known result for  PDEs to this class of PDAEs. These equations and hence the optimal control can be calculated without construction of the projected PDAE.    
Finally, we provide numerical simulations to illustrate application of the theoretical results.
\end{abstract}

\begin{keywords}
Partial differential-algebraic equations, linear-quadratic optimal control, differential Riccati equation, calculus of variation. 
\end{keywords}



\section{Introduction}

Partial differential-algebraic equations (PDAEs), also  known  as infinite-dimensional descriptors \cite{reis2008controllability},  singular distributed  parameter systems \cite{ge2009exact, ge2013well}, and abstract DAEs \cite{lamour2013differential} arise in many applications. Examples include fluid dynamics specifically  divergence-free fluid dynamics \cite{lin1997sequential},  electrical networks \cite{gunther2000joint}, nano-electronics \cite{bartel2008concept}, biology \cite{ahn2019global} and chemical engineering \cite{leung2013systems}. The well-posedness of linear PDAEs has received  attention, for example, \cite{sviridyuk2003linear, jacob2022solvability,erbay2024index,mehrmann2023abstract}  and the references therein. 


Various  definitions for the index of a PDAE exist in the literature\cite{campbell1999index, martinson2000differentiation, wagner2000further, rang2005perturbation}. Unlike (finite-dimensional) DAEs \cite{kunkel2006differential}, the different definitions of indices for PDAEs are not equivalent \cite{erbay2024index}. The concept of the index in PDAEs is crucial, serving as a key indicator of the numerical challenges expected when solving such equations. Moreover, it plays a  role in the assignment of initial and boundary conditions to ensure the consistency relation that results from the algebraic component of the system\cite{martinson2000differentiation}.  
Failure to meet this relation leads to distributions in the solution of the PDAE \cite{ge2019approximate}. 

Controller synthesis methods for PDAEs are needed in order stabilize their dynamics  and/or to  achieve   desired performance. One  valuable technique in  controller design is linear quadratic (LQ) control. Researchers have made progress in establishing LQ controller design for DAEs in finite-dimensional spaces, addressing index-1 \cite{bender1987linear,mehrmann1991autonomous} as well as  general higher-index \cite{reis2019linear,polezhaev2003spatial}. Using a singular value decomposition and calculus of variations, Bender and Laub \cite{1104694} studied the optimization problem for index-1 DAEs on a finite and infinite horizons. Moreover, to solve for the optimal control,  Bender and Laub \cite{1104694} have derived several differential  Riccati equations.  Mehrmann \cite{mehrmann1991autonomous} studied the finite-time LQ control of the same problem and showed that the existence of a unique continuous optimal control depends on the solvability of a two-point boundary value problem. This optimal control was assumed to satisfy  the  consistency relation on the initial condition.  On the other hand, Reis and Voigt \cite{reis2019linear}  used a behavior-based approach to study  optimal control for  DAEs with arbitrary index. Petreczky and Zhuk \cite{petreczkysolutions} also used behaviors to study optimal control for linear DAEs that are not regular.  

Optimal control for DAEs on infinite-dimensional spaces, that is, PDAEs, contains many open problems. In this context, we mention the work done by  Grenkin et al.  \cite{grenkin2016boundary} where they tackled the boundary optimal control problem  for a heat transfer model consisting of a coupled transient and steady-state heat equations. Grenkin et al.  \cite{grenkin2016boundary} showed the existence of weak solutions of this optimization problem under certain assumptions but without taking into consideration  the initial condition's consistency.   More recently,  Gernandt and Reis \cite{gernandt2024linear} studied   LQ optimal control for a class of PDAEs with resolvent index-one in the pseudo-resolvent sense by considering the mild solutions of the system.   Under specific conditions on the Popov operator, and provided that the initial algebraic sub-state of the system adheres to a certain algebraic constraint, Gernandt and Reis \cite{gernandt2024linear} showed the existence of unique minimizing control in the space of square integrable functions.    The optimal cost was shown to  be determined  by a bounded Riccati operator.

This paper studies  the LQ control problem over a finite-horizon for a class of linear PDAEs, those with  {\em radiality-index 0}.  (The radiality index will be defined in the next section.)  Many  equations arising in applications, such   a parabolic-elliptic systems and also the equation used to model the free surface of seepage liquid have  {\em radiality-index } 0 \cite{jacob2022solvability} have index zero.
With  proofs different from those for finite-dimensional systems  \cite{mehrmann1991autonomous, reis2019linear}, a fixed-point argument is used to show that there is a unique continuous optimal control, and this control can  written in feedback form.  We do not impose assumption on the initial values such as been done in \cite{gernandt2024linear}. Instead, by defining a set of admissible controls for the optimization problem, we demand that the control signal maintain the consistency of the initial conditions, thereby preventing distributions in the solutions \cite{ge2019approximate}. Next,  we derive a coupled system consisting of a differential Riccati-like equation and an algebraic equation that leads to the optimal control. It is important to note that specific projections are involved in the  derivation's proof process. However, after the state weight's decomposition through these projections is established, there is no  need to use these projections in computing the optimal control. This approach draws inspiration from the works of Heinkenschloss \cite{heinkenschloss2008balanced} and Stykel \cite{stykel2006balanced}, where they developed a Lyapunov equation within the study of balanced truncation model reduction for specific finite-dimensional systems. Similarly, Duan \cite{duan2010analysis}  derived a Lyapunov equation for a particular class of index-1 DAEs.  Breiten et al. \cite{breiten2019feedback} studied optimal control of the Navier-Stokes equations and derived a certain Riccati equation.  In fact, our work marks a first effort towards deriving differential Riccati-like equations for a general class of differential-algebraic equations, without needing to know the projection explicitly. We demonstrate that although  projections can be a valuable tool in deriving the desired equations, they may not be needed for computation of the control. Finally,  we illustrate the approach  by design of an LQ-optimal control for  an unstable coupled parabolic-elliptic system \cite{ParadaCerpaMorris,Alalabi&Morris}. 

The paper is structured as follows: In Section \ref{section1}, we formulate the problem, and define the class of PDAEs. Section \ref{section2} establishes the existence of an optimal control for the finite-horizon optimization problem. The derivation of a differential Riccati equation, essential for determining this optimal control, is elaborated  in Section \ref{section3}.  Finally, in Section \ref{section5}, numerical simulations are given to illustrate the theoretical results.

\section{Problem statement} \label{section1} Let $\mathcal{X},\; \mathcal{Z}$ and $\mathcal{U}$ be  Hilbert spaces.  Consider a system  modelled by
  \begin{subequations}\label{1}
\begin{align}
    \frac{d}{dt} Ex(t)=& A x(t)+Bu(t), \label{-1} \\
   Ex(0)=&Ex_i,\label{--1}
\end{align}
\end{subequations}
 where  $E \in \mathcal{L}(\mathcal{X}, \mathcal{Z}), \; A: D(A) \subset \mathcal{X} \rightarrow \mathcal{Z}$ is closed and densely defined on $\mathcal{X},$ $B \in \mathcal{L}(\mathcal{U}, \mathcal{Z})$. The initial condition    $x(0) = x_i \in \mathcal{X}$. Systems of the form \eqref{1} reduce to  classical infinite-dimensional systems  \cite{curtain2020introduction} when $E$ is  invertible.  The situation where  $E$ is non-invertible, and /or unbounded is of particular interest in this study. 
 
 It will be assumed throughout this paper that the  PDAE \eqref{--1} is {\em radial with degree 0} \cite{jacob2022solvability,sviridyuk2003linear}.  
 This implies  that exist projections  $P^{\mathcal{X}_1}$ and $P^{\mathcal{Z}_1}$ with ranges $\cX_1$ and $\cZ_1$ respectively, satisfying the  two assumptions below.   Let $\cX_0$ indicate the complement of $\cX_1$ and define similarly $\cZ_0, \cZ_1:$ 
\begin{align}
\mathcal{X}=\mathcal{X}_1 \oplus \mathcal{X}_0, \quad  \mathcal{Z}=\mathcal{Z}_1 \oplus \mathcal{Z}_0 .
\end{align} 
Use the projections to define
 \begin{subequations} \label{operators2}
\begin{align}
 &   E_0= E|_{\mathcal{X}_0}, \; \; & & E_1= E|_{\mathcal{X}_1}, &\; \; &  A_0= A|_{\mathcal{X}_0 \cap D(A)}, &\; \; &  A_1= A|_{\mathcal{X}_1 \cap D(A) },    \\
   & B_0= P^{\mathcal{Z}_0}B, & \; \; &   B_1= P^{\mathcal{Z}_1} B.   
\end{align}
\end{subequations}

\noindent \textbf{Assumption 1.}  
\begin{subequations} \label{relations}
        \begin{align}
& \text{For all } x \in \mathcal{X}, \; P^{\mathcal{Z}_1}Ex =EP^{\mathcal{X}_1}x,
\label{relation1}\\
& \text{For all } x \in D(A),\; P^{\mathcal{X}_1}x \in D(A)\; \text{and } P^{\mathcal{Z}_1}Ax= AP^{\mathcal{X}_1}x. \label{relation2}\end{align}
 \end{subequations}
 \noindent \textbf{Assumption 2.} 
 The operator $A_0 : D(A_0) \rightarrow \mathcal{Z}_0 $ has a bounded inverse: $   A_0^{-1} \in \mathcal{L}(\mathcal{Z}_0, \mathcal{X}_0), $ and the  operator \( E_0 \) is the zero operator. Also, $E_1 \in \mathcal{L}(\mathcal{X}_1, \mathcal{Z}_1)$ maps into $\mathcal{Z}_1$, and  has  a bounded inverse: $   E_1^{-1} \in \mathcal{L}(\mathcal{Z}_1, \mathcal{X}_1). $   The operator $A_1: D(A) \cap \mathcal{X}_1 \to \mathcal{Z}_1$ is closed and densely defined, and $ E_1^{-1} A_1 $ generates a $C_0-$semigroup operator $T(t)$ on $\mathcal{X}_1$.

Letting $I^{\cX}$ and $I^{\cZ}$ indicate  the identity operator on the  space $\mathcal{X}$ and $\mathcal{Z}$ , respectively, we define
\begin{align*}
   &P^{\mathcal{X}_0} = I^{\cX}-P^{\mathcal{X}_1}, \quad P^{\mathcal{Z}_0} = I^{\cZ}-P^{\mathcal{Z}_1}. 
\end{align*}
We also define
 \begin{eqnarray} 
   && \Tilde{P}^{\mathcal{X}}=\begin{bmatrix}
         P^{\mathcal{X}_1} \\P^{\mathcal{X}_0} 
    \end{bmatrix} \in \mathcal{L}(\mathcal{X},  \mathcal{X}_1 \times \mathcal{X}_0 ),    \quad  \Tilde{P}^{\mathcal{Z}}=\begin{bmatrix}
        P^{\mathcal{Z}_1} \\  P^{\mathcal{Z}_0} 
    \end{bmatrix} \in \mathcal{L}(\mathcal{Z}, \mathcal{Z}_1 \times \mathcal{Z}_0),
 \label{P&D} 
\end{eqnarray} 
\begin{eqnarray}
     && (\Tilde{P}^{\mathcal{X}})^{-1}=\begin{bmatrix}
         I^{\cX_1} & I^{\cX_0} 
    \end{bmatrix} \in \mathcal{L}(  \mathcal{X}_1 \times \mathcal{X}_0, \mathcal{X} ), \label{inverse-x}\\
    && (\Tilde{P}^{\mathcal{Z}})
    ^{-1}=\begin{bmatrix}
        I^{\cZ_1} & I^{\cZ_0} 
    \end{bmatrix} \in \mathcal{L}(  \mathcal{Z}_1 \times \mathcal{Z}_0, \mathcal{Z} ). \label{inverse-z}
\end{eqnarray}

  For convenience of notation, define
\begin{align}
    \Tilde{A}_1= E_1^{-1} A_1 ,  \; \; \;  \Tilde{B}_1= E_1^{-1} B_1, \; \; \;  \Tilde{B}_0 = A_0^{-1} B_0, \label{convenient-defns}
\end{align} 
and
\begin{align*}
    P^{\mathcal{X}_1} x_i = (x_i)_1, \quad  P^{\mathcal{X}_0} x_i = (x_i)_0 .
    \end{align*}

 Pre-multiply system \eqref{1} with the operator $\tilde{P}^{\mathcal{Z}}$ and  use   \eqref{relation1}-\eqref{relation2},  
\begin{subequations}  \label{composition} 
\begin{align}
      \frac{d}{dt}  x_1(t) =& \Tilde{A}_1 x_1(t)+ \tilde{B}_1 u(t),\; \;   x_1(0)  = (x_i)_1 ,  \label{a10} \\
    0 =&  x_0(t) +  \tilde{B}_0 u(t), \; \; \; \; x_0(0)=(x_i)_0  , \label{a11} 
\end{align}
\end{subequations}  
In common with finite-dimensions, this is known as  the Weierstra$\beta$  form; the systems considered here have degree of nilpotency 1 \cite{erbay2024index}.

 For continuous control  $u(t)\in C(0,t_f;\mathcal{U})$,  the mild solution  of \eqref{a10} is \cite[Chapter 5]{curtain2020introduction}
\begin{align}
     x_1(t)=  T(t) (x_i)_1 +  \int_{0}^{t} T(t-s) \Tilde{B}_1u(s) ds .\label{solutions-x1}
\end{align}
For system \eqref{a11} (see \cite{ge2019approximate}) 
\begin{align}
    x_0(t) = -\Tilde{B}_0u(t) + \delta(t) \Big((x_i)_0+\Tilde{B}_0u(0)\Big), \quad t\geq 0, \label{x0solution}
\end{align}
where $ \delta(t)$ is the Dirac delta function. If $(x_i)_0 = -\Tilde{B}_0u(0)$, then the distributional component of the solution \eqref{x0solution} is eliminated. 
 This equation is known as the consistency condition on the initial condition in the DAE literature \cite[Theorem 2.12]{kunkel2006differential}. With  a consistent initial condition 
\begin{align}
     x_0(t) = -\Tilde{B}_0u(t), \quad t\geq 0, \label{solutions-0-consistent}
\end{align}
and
\begin{align}
    x(t)&=- \Tilde{B}_0 u(t) +   T(t) (x_i)_1+  \int_{0}^{t} T(t-s) \Tilde{B}_1 u(s) ds . \label{solution-PDAE}
\end{align}

The optimal control problem  is to minimize, for  arbitrary initial condition \( x_i \),  the quadratic performance criterion 
\begin{align}
J(x_i,u ; t_f) =& \langle x(t_f), G x(t_f)\rangle_{\mathcal{X}}   + \int_{0}^{t_f} \langle x(s), Q x(s)\rangle_{\mathcal{X}}  \nonumber \\
& + \langle u(s), R u(s)\rangle_{\mathcal{U}} ds . \label{cost}
\end{align}
As usual, $R$ is assumed to be coercive, that is, $R$ is self-adjoint and $R\geq \epsilon I$ for some $\epsilon>0 $. Also, \( G \) and \( Q \) are self-adjoint non-negative operators. The  notation $\langle \cdot,\cdot\rangle _{\mathcal{U}}$ stands for the inner product on some Hilbert space $\mathcal{U}$.

In order to avoid distributions in the solution, the set of admissible controls for minimization of the cost $J(x_i, u;t_f)$
is
\begin{align}
 \mathcal{U}_{a} = \{ u(t) \in C([0,t_f];\mathcal{U}): \;  (x_i)_0 =-  \tilde{B}_0 u(0)\}. \label{admissible-set}
\end{align} 
This leads to a formal definition of the optimal control problem as
\begin{equation}
    \inf_{u \in \mathcal{U}_{a}} J(x_i, u; t_f), \label{optimal-control-problem}
\end{equation}
subject to \( x(t) \) that solves \eqref{1}.

The notation $P^{\mathcal{Z}_1*}$ and $P^{\mathcal{Z}_0*}$ shall denote the adjoint operators of  $P^{\mathcal{Z}_1}$ and $P^{\mathcal{Z}_0}$, respectively, and so
\begin{align*}
    &P^{\mathcal{Z}_1*} : \mathcal{Z} \to  \mathcal{Z}_1, \quad  P^{\mathcal{Z}_0*} : \mathcal{Z} \to  \mathcal{Z}_0,
\end{align*}
In a similar way, we define $P^{\mathcal{X}_1*}$ and $P^{\mathcal{X}_0*}$. From \eqref{relations}, we also have that
\begin{subequations} \label{relations-adjoint}
        \begin{align}
& \text{For all } z \in \mathcal{Z}, \; P^{\mathcal{X}_1*}E^* z =EP^{\mathcal{Z}_1*}z,
\\
& \text{For all } z \in D(A^*),\; P^{\mathcal{Z}_1*} z \in D(A^*)\; \text{and } P^{\mathcal{X}^*_1}A^* z = A^* P^{\mathcal{Z}_1*}z. 
\end{align}
 \end{subequations}
 
Throughout the  paper, it is assumed that the  final cost  vanishes on the subspace $\mathcal{X}_0$
    \begin{align}
&P^{{\mathcal{X}_1*}}  G P^{{\mathcal{X}}_0} = P^{{\mathcal{X}_0*}} G P^{{\mathcal{X}}_1}=P^{{\mathcal{X}_0*}} G P^{{\mathcal{X}}_0}=0 .\label{assumption1} 
\end{align}
Since in many applications 
$G=0$, this assumption is in practice not  difficult to meet. 
It is also assumed
\begin{align}
&    P^{{\mathcal{X}_0*}}Q P^{{\mathcal{X}_1}}  = P^{{\mathcal{X}_1*}}  Q P^{{\mathcal{X}_0}}= 0. \label{assumption2}
\end{align}
We also define the following restrictions on the subspaces $\mathcal{X}_1 $ and $\mathcal{X}_0$,
\begin{align*}
& G_1= G|_{\mathcal{X}_1} , \quad    Q_1= Q|_{\mathcal{X}_1},   \quad Q_0=Q|_{\mathcal{X}_0}. 
\end{align*}
Note that the operator $G_1$ maps to $X_1$ due to assumption \eqref{assumption1} and so $G_1 \in \mathcal{L}(\mathcal{X}_1)$. Similarly, \eqref{assumption2} implies that $Q_1\in \mathcal{L}(\mathcal{X}_1) $  and $Q_0\in \mathcal{L}(\mathcal{X}_0).$

\section{Existence of the optimal control} \label{section2}
Given the control problem defined in the previous section, and in particular the definition of the set of admissible controls in \eqref{admissible-set}, the set of admissible variations  $\mathcal{V}_{a}$, that is those for which if $u \in \mathcal{U}_a ,$ $v \in  \mathcal{V}_{a}$ then  $ u+v \in \mathcal{U}_a$ is
\begin{align}
 \mathcal{V}_{a} = \{ h(t) \in C([0,t_f];\mathcal{U}): \; 0 =-  \tilde{B}_0 h(0)\}. \label{admissible-variation}
\end{align} 

 The next proposition shows that  an arbitrary variation in the control leads to a change in cost that is a sum of quadratic and linear terms. 
 
\begin{prop} \label{proposition1}
 \label{Th1} Consider any $u(t) \in \mathcal{U}_a$, $x_i \in \mathcal{X}$ and $h(t) \in  \mathcal{V}_{a}$. Define,
 letting $T^*(t)$ indicate the $C_0$-semigroup on $\mathcal{Z}_1$ generated by  $\Tilde{A}_1^*,$
\begin{align}
    z_u(t)= & 2 \Big[ \tilde{B}_1^*  T^*(t_f-t) G_1 x_1(t_f) + \tilde{B}_1^*  \int_t^{t_f}  T^*(r-t)  Q_1 x_1(r) dr \nonumber\\
 &-\tilde{B}_0^* Q_0 x_{0}(t)  + R u(t) \Big], \label{zu}
\end{align}
\begin{align}
        \phi(t) h(t)  = \int_{0}^{t} T(t-s) \Tilde{B}_1h(s) ds 
         -   \tilde{B}_0 h(t) . \label{phi-}
\end{align}  
Then  variation 
\begin{align}
     J(x_i,u+h ; t_f) - & J(x_i,u ; t_f) =  \int_0^{t_f} \langle z_u(s),h(s)\rangle_{\mathcal{U}}   ds   \nonumber \\
    & +  \langle \phi(t_f)h(t_f), G \phi(t_f)h(t_f)\rangle_{\mathcal{X}}    +\int_0^{t_f} \langle \phi(s)h(s), Q \phi(s)h(s)\rangle_{\mathcal{X}}    \nonumber \\
    & + \langle h(s), Rh(s)\rangle_{\mathcal{U}}   ds. \label{difference}
\end{align}

\end{prop}
\begin{proof}
     The proof is similar  to the classical argument in $LQ-$optimal control for ODEs e.g. \cite{morris2000introduction}.
Let $u(t) \in \mathcal{U}_a$ and let $h(t)\in \mathcal{V}_a$. Then, $(u+h)(t) \in \mathcal{U}_a$ since  we have that  $(u+h)(t) \in C(0,t_f; \mathcal{U})$, and
\begin{align}
   -  \tilde{B}_0 (u+h)(0)  =  (x_i)_0.
\end{align}
Let $x(t)$ indicate the trajectory corresponding to $u(t)$, then from \eqref{solution-PDAE}
\begin{align}
    x(t) = -\tilde{B}_0 u(t)+  T(t) (x_i)_1 +   \int_{0}^{t} T(t-s) \Tilde{B}_1u(s) ds . 
    \end{align}
With the definition of $\phi(t) h(t)$ in \eqref{phi-}, the trajectory corresponding to $(u+h)(t)$ is 
\begin{align}
    x_{u+h}(t) = x(t) + \phi(t) h(t).
\end{align}
 Thus, 
\begin{align}
     J(x_i, u+h;t_f) &-  J(x_i, u;t_f)
    =  2 \langle x(t_f), G \phi(t_f) h(t_f)\rangle_{\mathcal{X}}   + 2 \Big( \int_0^{t_f}  \langle u(s), Rh(s)\rangle_{\mathcal{U}}     \nonumber \\
    &+ \langle x(s), Q \phi(s)h(s)\rangle_{\mathcal{X}}   ds \Big) + \langle \phi(t_f) h(t_f), G \phi(t_f) h(t_f)\rangle_{\mathcal{X}}  \nonumber \\
    &+ \int_0^{t_f} \langle \phi(s)h(s), Q \phi(s)h(s)\rangle_{\mathcal{X}}   + \langle h(s), Rh(s)\rangle_{\mathcal{U}}   ds. 
\label{difference-2}
\end{align}
Setting  
\begin{align}
    \nabla J(u,h,t_f) & = 2 \langle x(t_f), G \phi(t_f) h(t_f)\rangle_{\mathcal{X}}   + 2 \Big( \int_0^{t_f} \langle x(s), Q \phi(s)h(s)\rangle_{\mathcal{X}}   \nonumber \\
    &+ \langle u(s), Rh(s)\rangle_{\mathcal{U}}   ds \Big), 
\end{align}
we rewrite \eqref{difference-2} as
\begin{align}
     J(x_i,u+h ; t_f) - J(x_i,u ; t_f)& =\nabla J(u,h,t_f) + \langle \phi(t_f) h(t_f), G \phi(t_f) h(t_f)\rangle_{\mathcal{X}}  \nonumber \\
    &+ \int_0^{t_f} \langle \phi(s)h(s), Q \phi(s)h(s)\rangle_{\mathcal{X}}   + \langle h(s), Rh(s)\rangle_{\mathcal{U}}   ds. \label{part4}
\end{align}
In the following steps, we compute a $z_u \in L_2(0,t_f; \mathcal{U})$ in a way that 
\begin{align}
    \nabla J (u,h,t_f) = \int_0^{t_f} \langle z_u(s) ,h(s)\rangle_{\mathcal{U}}   ds.
\end{align}
Using the definition of $\phi(t) h(t)$ and the assumptions \eqref{assumption1} and \eqref{assumption2} on the operators $G$ and $Q$, we write
\begin{align}
    & \nabla J(u,h,t_f) = 2 \Big[ \langle  x_1(t_f), G_1 \int_{0}^{t_f} T(t_f-s)   \tilde{B}_1 h(s) ds \rangle_{\mathcal{X}}   +\int_0^{t_f} \langle u(s) , Rh(s)\rangle_{\mathcal{U}}    \nonumber\\
    & +  \langle x_1(s), Q_1  \int_{0}^{s} T(s-r)  \tilde{B}_1 h(r) dr  \rangle_{\mathcal{X}}   -    \langle x_0(s), Q_0   \tilde{B}_0  h(s)\rangle_{\mathcal{X}}    ds \Big]. \nonumber \\
    \label{eq25}
\end{align}
To simplify the right-hand-side of \eqref{eq25}, we note that after straightforward calculations
\begin{align}
   &  \langle x_1(t_f), G_1  \int_{0}^{t_f} T(t_f-s)   \tilde{B}_1 h(s) ds\rangle_{\mathcal{X}}    +  \int_0^{t_f} \langle x_1(s), Q_1 \nonumber \\
   &  \int_{0}^{s} T(s-r)  \tilde{B}_1  h(r) dr \rangle_{\mathcal{X}}   ds  =   \int_0^{t_f}   \langle  \tilde{B}_1^* T^*(t_f-s)  G_1 x_1(t_f) \nonumber \\
   & + \tilde{B}_1^* \int_s^{t_f} T^*(r-s)  Q_1 x_1(r)dr, h(s)\rangle_{\mathcal{U}}   ds  , \label{part1}
\end{align}
and
\begin{align}
     -2 \int_0^{t_f}\langle   x_0(s), Q_0  \tilde{B}_0 h(s)\rangle_{\mathcal{X}}   ds= - 2  \int_0^{t_f}\langle   \tilde{B}_0^*  Q_0x_0(s), h(s)\rangle_{\mathcal{U}}   ds  . \label{part2}
\end{align}
Subbing \eqref{part1} and \eqref{part2} in \eqref{eq25}, we arrive at 
\begin{align}
    \nabla J (u,h,t_f) &= \int_{0}^{t_f} \Big[ \langle  2\tilde{B}_1^*  T^*(t_f-s)   G_1 x_1(t_f)  \nonumber \\
    &  +  2 \tilde{B}_1^* \int_s^{t_f} T^*(r-s) Q_1x_1(r)dr -   2 \tilde{B}_0^* Q_0x_0(s) + 2 Ru(s) ,   h(s)\rangle     \Big] ds.\label{part5}
\end{align}
Defining  $z_u(t)$ as given in \eqref{zu}, the conclusion follows by combining \eqref{part4} and \eqref{part5}.
\end{proof}

\smallskip

The next theorem shows the optimal control problem  \eqref{optimal-control-problem} has a solution.

\begin{theorem} \label{fixed-point}
    The control input
    \begin{align}
        u^{opt}(t) =& -R^{-1}  \Big[ \tilde{B}_1^* T^*(t_f-t)  G_1 x^{opt}_1 (t_f) + \tilde{B}_1^* \int_t^{t_f}  T^*(r-t)  Q_1 x^{opt}_1(r) dr \nonumber\\
 &- \tilde{B}^*_0 Q_0 x^{opt}_{0}(t) \Big], \label{optimal control-feedback}   
    \end{align}
    is the unique solution of the optimization problem \eqref{optimal-control-problem}. Here $ x^{opt}(t)$ is the system state corresponding to  control $ u^{opt}(t)$ , and $x^{opt}(t) = x_1^{opt}(t) + x_0^{opt}(t)$. 
\end{theorem}
\begin{proof}The proof takes two main steps.  First, we show that $u^{opt}(t)$  is the unique solution in $C(0,t_f;\mathcal{U})$ solving equation 
\begin{align}
    z_u(t)=0. \label{fixed-point-argument}
\end{align}
 We rewrite \eqref{fixed-point-argument} in terms of a fixed-point operator. Define  the  operator
    \begin{align}
        (V_1 u)(t) =  T(t) (x_i)_1 +  \int_0^t T(t-s) \tilde{B}_1 u(s) ds. \label{V-1}
    \end{align}
It is straightforward to show that $V_1u(t):C([0,t_f];\mathcal{U})   \to C([0,t_f];\mathcal{U}),$ by using the fact that $T(t)$ is a $C_0$-semigroup, and following a similar argument as the one given in \cite[Lemma 5.1.5]{curtain2020introduction}.
Next,  referring to \eqref{zu}, we rewrite equation \eqref{fixed-point-argument} 
by using   \eqref{solutions-x1} and \eqref{solutions-0-consistent}
\begin{align*}
    Ru(t) &=   -\tilde{B}_1^*  T^*(t_f-t) G_1 x_1 (t_f)  -\tilde{B}_1^* \int_t^{t_f}  T^*(r-t)  Q_1 \Big(
 T(r) (x_i)_1  \nonumber \\
&+  \int_0^r T(r-s)  \tilde{B}_1 u(s) ds \Big) dr -\tilde{B}^*_0 Q_0  A_0^{-1} B_0 u(t). 
\end{align*}
The previous equation can be now written with the help of operator $V_1$ as

\begin{align*}
    Ru(t) &=   -\tilde{B}_1^* T^*(t_f-t) G_1 x_1 (t_f)  - \tilde{B}_1^*  \int_t^{t_f}  T^*(r-t)   Q_1 (V_1u)(r) dr \nonumber\\
 &-\tilde{B}^*_0 Q_0  A_0^{-1} B_0 u(t) , 
\end{align*} 
that is 
\begin{align}
  (R+ \Tilde{B}^*_0  Q_0  \Tilde{B}_0)  u(t) &= -  \tilde{B}_1^* T^*(t_f-t)  G_1 x_1 (t_f) \nonumber \\
  & - \tilde{B}_1^* \int_t^{t_f}  T^*(r-t)    Q_1 (V_1u)(r) dr . \label{input-new}
\end{align}
Since the operator $ \Tilde{B}^*_0  Q_0  \Tilde{B}_0$ is positive semi-definite and $R$ is coercive, it follows that $R+ \Tilde{B}^*_0  Q_0  \Tilde{B}_0$ is also coercive, and so has a bounded inverse. For convenience of notation, we set 
\begin{align}
    \tilde{R}=R+ \Tilde{B}^*_0  Q_0  \Tilde{B}_0.
\end{align}
For $u(t) \in  C([0,t_f];\mathcal{U}) $,  define  $F:C([0,t_f];\mathcal{U})   \to C([0,t_f];\mathcal{U}) $ as
\begin{align}
(F u)(t) &= -  (\tilde{R})^{-1}  \Tilde{B}_1^* T^*(t_f-t) G_1 x_1 (t_f) \nonumber \\
   & - (\tilde{R})^{-1} \Tilde{B}_1^* \int_t^{t_f}  T^*(r-t)   Q_1 (V_1u)(r) dr .
\end{align}
With this definition,    equation \eqref{input-new} can be written
\begin{align}
    u(t)=(F u)(t). \label{F-eq}
\end{align}
It will now be shown that $F^n$ is a contraction for large enough $n$.   For $u_1, \; u_2 \in C(0,t_f; \mathcal{U})$
\begin{align*}
&|(Fu_1)(t) - (Fu_2)(t)| \\
&= |(\tilde{R})^{-1} \Tilde{B}_1^* 
\int_{t}^{t_f} T^*(r-t)  Q_1 \Big[ (V_1u_1)(r)  - (V_1u_2)(r) \Big] \, dr| 
\\
& \leq \|(\tilde{R})^{-1}  \Tilde{B}_1^*\| 
\int_{t}^{t_f} \| T^*(r-t) \|   \|  Q_1\|  |  (V_1u_1)(r)  - (V_1u_2)(r)|   dr \nonumber \\
 &  \leq \|(\tilde{R})^{-1}  \Tilde{B}_1^*\| 
\int_{t}^{t_f} \| T^*(r-t) \|   \| Q_1\|   \int_0^{r}  \|T(r-s)\| \| \tilde{B}_1\|  |u_1(s) -u_2(s) |  ds   dr . \nonumber 
 \end{align*}
Since $T(t)$ is bounded on every finite subinterval of $[0, \infty)$; see \cite[Theorem 2.1.7 a]{curtain2020introduction}, and so is $T^*(t)$,  the operators $(\Tilde{R})^{-1} \Tilde{B}_1^* =   ( R+ \Tilde{B}^*_0  Q_0  \Tilde{B}_0)^{-1} \Tilde{B}_1^* \in \mathcal{L}(\mathcal{X}_1, \mathcal{U}) $,  and   $ Q_1$, $\tilde{B}_1$ are also bounded linear operators, there is  a constant $M>0$ such that 
\begin{equation*}
    |(Fu_1)(t) - (Fu_2)(t)| 
  \leq M  \int_t^{t_f} \int_0^{r} |u_1(s) -u_2(s) |  ds dr .
\end{equation*}
 Thus,
 \begin{align*}
|(Fu_1)(t) - (Fu_2)(t)| 
 & \leq M \| u_1-u_2 \|_{\infty} \int_t^{t_f} \int_0^{r} ds dr \nonumber \\ & \leq  t_f (t_f-t) M \| u_1-u_2 \|_{\infty}.
\end{align*}
 Replace $u_1$ and $u_2$ in the inequalities above with $Fu_1$ and $Fu_2$, respectively,
\begin{align}
     |(F^2u_1)(t) - (F^2u_2)(t)|& \leq M \int_{t}^{t_f}\int_0^{r} |(Fu_1)(s)-(Fu_2)(s)| ds dr \nonumber \\
     &  \leq  M^2 t_f \|u_1-u_2 \|_{\infty}  \int_{t}^{t_f} \int_0^{r} (t_f-s) ds dr \nonumber \\
     & = M^2 t_f \|u_1-u_2 \|_{\infty}  \int_{t}^{t_f} \frac{(t_f-r)^2 }{2} dr \nonumber \\
        & \leq  M^2  t_f^2  \|u_1-u_2 \|_{\infty}  \frac{(t_f-t)^2 }{2} .
      \end{align}
An  induction argument shows that
\begin{align}
    |(F^nu_1)(t) - (F^nu_2)(t)|& \leq M^n t_f^n \frac{(t_f-t)^{n}}{n!}   \|u_1-u_2 \|_{\infty}.
\end{align}
Defining $a=(M t_f^2),$ it follows that
\begin{align}
    \|(F^nu_1)(t) - (F^nu_2)(t)\|_\infty& \leq   \frac{a^n}{n!}  \|u_1-u_2 \|_{\infty}.
\end{align}

For sufficiently large $n$, $\frac{a^n}{n!} <1$ and hence $F^n$
is  a contraction for large enough $n.$ It follows from  \cite[Lemma 5.4-3]{kreyszig1991introductory} that $F$ has a unique fixed point.  Hence, there exists a unique control  in $ C(0,t_f; \mathcal{U})$ that solves equation \eqref{F-eq}, and is also  the unique solution of equation \eqref{fixed-point-argument}.

The control $u^{opt}(t) \in C(0,t_f;\mathcal{U})$, and is  the unique control leading to $z_u=0$. Since equation \eqref{difference} was derived by allowing for the variation $h(t) \in \mathcal{V}_a$,  ensuring  that $(u+h)(t)$ satisfies the consistency condition, it follows that $u^{opt}(t)$ ensures the consistency of the initial condition, that is, $u^{opt}(t) \in \mathcal{U}_a$. 

Referring to  \eqref{difference}, since $R>0, $ $Q\ge0, $  $G \geq 0,$ for any admissible variation $h,$ 
\begin{align}
    J(x_i,u^{opt}+h;t_f) - J(x_i,u^{opt};t_f) > 0 ,
\end{align}
and so $u^{opt} (t)$ is the unique control minimizing $J.$
\end{proof}

 \section{  Derivation of differential Riccati equations} \label{section3}
Consider  the cost functional \eqref{cost} with a variable initial
time $t_0 $, $ 0 \leq t_0 \leq t_f$,
\begin{align}
J(x_i,u ; t_0, t_f) =& \langle x(t_f), G x(t_f)\rangle_{\mathcal{X}}   + \int_{t_0}^{t_f} \langle x(s), Q x(s)\rangle_{\mathcal{X}}  \nonumber \\
& + \langle u(s), R u(s)\rangle_{\mathcal{U}} ds . \label{cost2}
\end{align}
Since the governing equations are time-invariant, assuming  $u(t_0)$ is consistent with  $x(t_0)$, the previous section's results apply.  
Thus, there is unique input that minimizes the cost functional \eqref{cost2} for trajectories of \eqref{1} with initial condition $x(t_0) = x_i$.  The  control that minimizes  cost functional \eqref{cost2} is denoted $u^{opt}(\cdot; x_i, t_0, t_f)$, and its corresponding optimal state trajectory is  $x^{opt}(\cdot; x_i, t_0, t_f).$ The control that minimizes
\eqref{cost} shall be denoted by $ u^{opt}(t)$ or $u^{opt}(\cdot; x_i, 0, t_f)$, and its corresponding state trajectory is $x^{opt}(t)$ or  $x^{opt}(\cdot; x_i, 0, t_f)$. 

The following result holds due to the uniqueness of the optimal control. It is an extension of the principle of optimality from linear PDEs to linear PDAEs, and for each of the sub-states.

\begin{lem} \label{opt-principle} Let  $0 \leq t_0 \leq t_f$.   For all $ s \in [t_0, t_f]$, 
    \begin{align}
x^{opt}(s; x_i, 0, t_f) = x^{opt}(s; x^{opt}(t_0, x_i, 0, t_f), t_0, t_f). \label{equation001}
\end{align} 
In addition, each of the dynamical and the algebraic  sub-states satisfy the principle of optimality, that is, 
\begin{align*}
&x_{0}^{opt}(s; (x_i)_0, 0, t_f) = x_0^{opt}(s;  x^{opt}_0(t_0, (x_i)_0, 0, t_f), t_0, t_f),
\end{align*}
and
\begin{align}
&x_{1}^{opt}(s; (x_i)_1,0, t_f) = x_1^{opt}(s; x^{opt}_1(t_0, (x_i)_1, 0, t_f), t_0, t_f). \label{x-1-relation}
\end{align}
\end{lem}
\begin{proof}
Since the optimal control $u^{opt}(t)$ is unique,   we can use a similar line of reasoning as the one presented in  \cite[Lemma 9.1.7]{curtain2020introduction} to  establish equation \eqref{equation001}  for all $ s \in [t_0, t_f]$.  Using once more the uniqueness of the optimal control and also equation \eqref{a11}, we have for any $ s \in [t_0, t_f]$ 
\begin{align}
     x_0^{opt}(s; (x_i)_0, 0, t_f)  &= \Tilde{B}_0 u^{opt}(s; x_i, 0, t_f) \nonumber \\
    & = \Tilde{B}_0 u^{opt}(s; x^{opt}(t_0, x_i, 0, t_f), t_0, t_f) \nonumber \\
    &= x_0^{opt}(s; x^{opt}_0(t_0, (x_i)_0, 0, t_f), t_0, t_f) .\label{equation2}
\end{align}
 Thus, the algebraic sub-state of the PDAE  conforms to the principle of optimality. The trajectory of system \eqref{1}  on $[0,t_f]$ with initial condition $x_i$ is  
    \begin{align}
x^{opt}(s; x_i, 0, t_f)&  = x_1^{opt}(s; (x_i)_1, 0, t_f) + x_0^{opt}(s; (x_i)_0, 0, t_f) , \label{equation1} 
\end{align}
and the trajectory   on $[t_0,t_f]$ with initial condition $x^{opt}(t_0, x_i, 0, t_f)$ is  
\begin{align}
x^{opt}(s; x^{opt}(t_0, x_i, 0, t_f) , t_0, t_f)&  = x_1^{opt}(s; x_1^{opt}(t_0, (x_i)_1, 0, t_f), t_0, t_f) \nonumber \\
& + x_0^{opt}(s; x^{opt}_0(t_0, (x_i)_0, 0, t_f), t_0, t_f). \label{equation10} 
\end{align}    
Using \eqref{equation001}, we deduce that the right-hand-side of equations \eqref{equation1} and \eqref{equation10} are equal. It then follows from \eqref{equation2} that the dynamical sub-state of the PDAE, i.e. $x^{opt}_1$,  also  satisfies the principle of optimality \eqref{x-1-relation}.
\end{proof}
The next proposition demonstrates that at any given time $t \in [0,t_f]$, the optimal control \eqref{optimal control-feedback} can be written as a feedback of the dynamical state $x^{opt}_1$.

\begin{prop}   The optimization problem \eqref{optimal-control-problem} reduces to   minimizing  the cost functional
\begin{align}
    J((x_i)_1,u ; t_f) =& \langle x_1 (t_f), G_1 x_1(t_f)\rangle_{\mathcal{X}}   + \int_{0}^{t_f} \langle x_1(s), Q_1 x_1(s)\rangle_{\mathcal{X}}  \nonumber \\
    &+ \langle u(s), \Tilde{R} u(s)\rangle_{\mathcal{U}} ds , \label{cost-dynamical}
\end{align}
over the set of admissible control $\mathcal{U}_a$, where $x_1(t)$  solves  system \eqref{a10}. 

The minimizing optimal control $u^{opt}(t)$ can be rewritten as 
      \begin{align}
        u^{opt}(t) =&  - \Tilde{R}^{-1} \tilde{B}_1^*\big[   T^*(t_f-t)  G_1 x^{opt}_1 (t_f) + \int_t^{t_f}  T^*(r-t)  Q_1 x^{opt}_1(r) dr \big]. \label{optimal control-dynamical state}
    \end{align}
\end{prop}
\begin{proof}   Rewrite  the cost functional \eqref{cost} as
  \begin{align}
    J(x_i,u ; t_f) =& \langle x_1 (t_f), G_1 x_1(t_f)\rangle_{\mathcal{X}}   + \int_{0}^{t_f} \langle x_1(s), Q_1 x_1(s)\rangle_{\mathcal{X}}  + \langle x_0(s), Q_0 x_0(s)\rangle_{\mathcal{X}}  \nonumber \\
& + \langle u(s), R u(s)\rangle_{\mathcal{U}} ds .
\end{align}
Since the optimization is over $\mathcal{U}_a$, use that $x_0^{opt}(t) = - \Tilde{B}_0 u^{opt}(t)$ to obtain the cost \eqref{cost-dynamical}. Note that the  existence of unique optimizing control for cost functional  \eqref{cost-dynamical} over the set $\mathcal{U}_a$ follows from the results in section\ref {section2} concerning the equivalent optimization problem \eqref{optimal control-feedback}. To prove statement \eqref{optimal control-dynamical state}, we first recall that $\Tilde{R}=R+ \Tilde{B}^*_0  Q_0  \Tilde{B}_0 $. From \eqref{optimal control-feedback}, it follows that
\begin{align*}
    u^{opt}(t) =& -R^{-1}  \Big[ \tilde{B}_1^* T^*(t_f-t)  G_1 x^{opt}_1 (t_f) + \tilde{B}_1^* \int_t^{t_f}  T^*(r-t)  Q_1 x^{opt}_1(r) dr \nonumber\\
 &+ \tilde{B}^*_0 Q_0 \Tilde{B}_0 u^{opt}(t)\Big].
\end{align*}
 Now the previous equation can be rewritten as
\begin{align*}
    (R+\tilde{B}^*_0 Q_0 \Tilde{B}_0)  u^{opt}(t) =& -  \tilde{B}_1^* \Big[ T^*(t_f-t)  G_1 x^{opt}_1 (t_f) + \int_t^{t_f}  T^*(r-t)  Q_1 x^{opt}_1(r) dr \Big].
\end{align*}
 Since the operator  $\Tilde{R}=R+ \Tilde{B}^*_0  Q_0  \Tilde{B}_0$ is  coercive, we arrive at equation \eqref{optimal control-dynamical state}. 
\end{proof}
The optimization problem \eqref{optimal-control-problem} simplifies to a standard LQ problem on $\mathcal{X}_1$. This simplification allows us to apply well-known results on the finite-time LQ problem; see  \cite[Chapter 9]{curtain2020introduction}. 

\begin{lem} \cite[Lemma 9.1.9]{curtain2020introduction} \label{projections}  For any $t \in [0,t_f]$ and any $x_1 \in \mathcal{X}_1$,   define the operator  $\Pi_1(t)$ on $\mathcal{X}_1$  
\begin{align}  
    \Pi_1(t)  x_1 & =       T^*(t_f-t)  G_1 x_1^{opt}(t_f; x_1, t, t_f)  \nonumber\\
 & +  \int_{t}^{t_f}  T^*(r-t)    Q_1 x_1^{opt}(r; x_1 , t, t_f) dr ; \label{pi1} 
\end{align}
 $ \Pi_1(t) \in \mathcal{L}(\mathcal{X}_1) $ for all $t\in [0,t_f]$.   The optimal control \eqref{optimal control-dynamical state} can be written as 
        \begin{align}
   u^{opt}(t; x_i, 0, t_f) 
   =&-\Tilde{R}^{-1}  \Tilde{B}_1^* \Pi_1(t)    x_1^{opt} (t; (x_i)_1, 0, t_f)  , \label{optimal-control-x1-pi1}
\end{align} 
and the minimum cost is
 \begin{align}
    &  J(x_i,u^{opt} ; t_f)  =   \langle  (x_i)_1 ,   \Pi_1(0)  (x_i)_1\rangle _{\mathcal{X}}  .
\label{min-cost-E}
\end{align}
\end{lem}

 The  optimal dynamical state $x^{opt}_1(t)$ is   the mild solution to an abstract evolution equation. This is stated in the corollary below.
\begin{cor} \label{mild evolution operator} 
The optimal sub-state $x^{opt}_1(t)$  is the mild solution of the abstract evolution equation
\begin{equation}
    \begin{aligned}
    \frac{d}{dt}  x^{opt}_1(t)   &=  (E_1^{-1} A_1 - \tilde{B}_1\Tilde{R}^{-1}  \tilde{B}_1^* \Pi_1(t)    ) x^{opt}_1 (t) ,  \;   \\
     x_1^{opt} (0) &=   (x_i)_1 . 
\end{aligned} \label{evolution-eq-x1}
\end{equation}
Also, the operator $ E_1^{-1} A_1 - \tilde{B}_1\Tilde{R}^{-1}  \Tilde{B}_1^* \Pi_1(t)   $    generates the mild evolution operator $U(t, s)$ on the set $\{(t, s); 0 \leq s \leq t \leq t_f\}$, so 
\begin{align}
    x^{opt}_1(t;(x_i)_1, t_0 , t_f) = U(t,t_0) (x_i)_1.  \label{solutin-x1-opt}
\end{align}
\end{cor}
\begin{proof}
    This result follows directly by using  the expression of the optimal control in \eqref{optimal control-dynamical state}, and applying the results in \cite[Corollary 9.1.10]{curtain2020introduction}.
\end{proof}
As a matter of fact, the operator-valued function  $ \Pi_1(t)$ is the unique solution of a standard differential Riccati equation.
\begin{lem} \label{lemma-self-adjoint-pi1}\cite[Lemma 4.3.2,Theorem 9.1.11]{curtain2020introduction} The operator-valued function $ \Pi_1(t)$ solves the following differential Riccati equation
\begin{align}
     \frac{d}{dt}    \Pi_1(t) x_1  +  \Pi_1(t)   E_1^{-1} A_1 x_1 +  A_1^* E_1^{-*}\Pi_1(t)   x_1& \nonumber \\
     -  \Pi_1(t)   \Tilde{B}_1   \Tilde{R}^{-1} \tilde{B}_1^* \Pi_1(t)  x_1  +     Q_1x_1 &=0 , \qquad \forall x_1 \in D(A_1) , \label{standard DRE} \\
 \Tilde{\Pi}_1(t_f) x_1 &= G_1x_1   . \label{final-pi1}
\end{align} 
The  operator-valued function $\Pi_1(t)$ is the unique solution of this equation in the
class of strongly continuous, self-adjoint operators in $\mathcal{L}(\mathcal{X}_1)$ such that $\langle  x_1^a,  \Pi_1(t) x^b_1 \rangle_{\mathcal{X}} $ is differentiable for $t \in (0,t_f) $ and $ x_1^a, \;  x_1^b \in D(A_1)$. 
\end{lem}
This leads to the main result of this section: a characterization of  the optimal control \eqref{optimal control-feedback} that does not require calculating the restrictions of operators  $A, E $ on the subspace $\mathcal{X}_1$.



\begin{theorem} \label{thm-general-p} Define the operator
\begin{align}
    \Pi_0   &=   -  A_0^{-*} Q_0 ; \label{pi0}  
\end{align}
$\Pi_0 \in \mathcal{L}(\mathcal{X}_0, \mathcal{Z}_0) $. Also, recalling  the operator-valued function \eqref{pi1}, define 
\begin{subequations} \label{tilde}
    \begin{align}
   &\Tilde{\Pi}_1(t)=  P^{\mathcal{Z}_1*} E_1^{-*} \Pi_1(t) E_1^{-1} P^{\mathcal{Z}_1}  , \label{pi tilde 1}\\
    &  \Tilde{\Pi}_0=  P^{\mathcal{Z}_0*}  \Pi_0 P^{\mathcal{X}_0}  .\label{pi tilde 0}
\end{align}
\end{subequations}

\begin{enumerate}
    \item The solution of the optimization problem \eqref{optimal-control-problem} is 
    \begin{align}
   u^{opt}(t) 
   =&-R^{-1} B^* (\Tilde{\Pi}_0 + \Tilde{\Pi}_1(t) E)  x^{opt}(t). \label{optimal-control-riccati-no-projection}
\end{align} 
\item  The operator   $\Tilde{\Pi}_0$ solves the algebraic equation
\begin{align}
      A^*  \Tilde{\Pi}_0 x &=- Q P^{\mathcal{X}_0} x  , \quad \forall x\in \mathcal{X},
\label{equation-pi0-no-projection} 
\end{align}
and is uniquely defined on $\mathcal X_0 .$
\item
In addition,    $ \Tilde{\Pi}_1(t) \in  C ([0,t_f]; \mathcal{L}(\mathcal{Z}))$   solves  
    \begin{subequations} \label{riccati-eqs-no-projection}
\begin{align}
 \frac{d}{dt}  E^*  \Tilde{\Pi}_1(t)  E x  + E^*  \Tilde{\Pi}_1(t)   Ax +   A^* \Tilde{\Pi}_1(t)   E x& - E^* \Tilde{\Pi}_1(t)    B   R^{-1} B^* \Tilde{\Pi}_1(t)   E x  \nonumber \\
   -E^* \Tilde{\Pi}_1(t)    B   R^{-1} B^*  \Tilde{\Pi}_0x +     Q P^{\mathcal{X}_1} x &=0 , \qquad \forall x \in D(A) , \label{equation-pi1-no-projection} \\
    E^* \Tilde{\Pi}_1 (t_f) E x&= G x   . \label{final-pi1-no-projection} 
\end{align}  
\end{subequations} 
such that $\langle  Ex^a,  \Tilde{\Pi}_1(t) E x^b \rangle_{\mathcal{Z}} $ is differentiable for $t \in (0,t_f) $ and $x^a,x^b \in D(A).$ 
\item  The minimum cost is
 \begin{align}
    &  J(x_i,u^{opt} ; t_f)  =   \langle  Ex_i ,   \Tilde{\Pi}_1(0) Ex_i\rangle _{\mathcal{Z}}  . 
\label{min-cost-no-projection}
\end{align}
 \end{enumerate}
\end{theorem}
\begin{proof}
\noindent \textbf{1.} Using the definitions of the operators  $ \Tilde{B}_0 $ and $ \Tilde{R}$ ,  write the optimal control \eqref{optimal-control-x1-pi1} as follows:
\begin{align*}
  (R+  B^*_0  A_0^{-*}  Q_0   \Tilde{B}_0)  u^{opt}(t) =& -  \tilde{B}_1^* \Pi_1(t)   x_1^{opt}(t),
  \end{align*}
and so
  \begin{align*}
  u^{opt}(t) &= - R^{-1} \Big(\tilde{B}_1^* \Pi_1(t)   x_1^{opt}(t) + B^*_0  A_0^{-*}  Q_0  \Tilde{B}_0  u^{opt}(t)\Big) . \end{align*} 
Since $u^{opt}(t) \in \mathcal{U}_a$, we  use that $x_0^{opt}(t) = - \Tilde{B}_0 u^{opt}(t)$  to rewrite the previous equation
  \begin{align}
   u^{opt}(t) &= - R^{-1} \Big( \tilde{B}_1^* \Pi_1(t)   x_1^{opt}(t) -   B^*_0  A_0^{-*} Q_0   x_0^{opt}(t) \Big)  \nonumber \\
   &=  - R^{-1} \Big(\tilde{B}_1^* \Pi_1(t)   x_1^{opt}(t) + B^*_0  \Pi_0   x_0^{opt}(t) \Big),  \label{optimal-control-pi1e-pi0} \end{align}
   where we also used to the definition of operator $\Pi_0$ in \eqref{pi0}.  From   \eqref{relations}a, and \eqref{operators2}, recall that for any $x(t) \in \mathcal{X}$
   \begin{subequations} \label{part1-defintions}
           \begin{align}
    &E_1 x_1(t) = E P^{\mathcal{X}_1} x(t) = P^{\mathcal{Z}_1} E x(t), \\
  & B_1=P^{\mathcal{Z}_1}B, \quad B_0=P^{\mathcal{Z}_0} B .\end{align}
     \end{subequations}
 Using the  statements above and that $\tilde{B}_1=E_1^{-1} B_1 $,  rewrite \eqref{optimal-control-pi1e-pi0} as
\begin{align*}
    u^{opt}(t) &=-R^{-1} B^* \Big(  P^{\mathcal{Z}_1*} E_1^{-*} \Pi_1(t) E_1^{-1}  E P^{\mathcal{X}_1}   x^{opt} (t) +  P^{\mathcal{Z}_0*} \Pi_0  P^{\mathcal{X}_0} x^{opt}  (t)  \Big).
\end{align*} 
Equation \eqref{optimal-control-riccati-no-projection} now follows  from the definition of  $\Tilde{\Pi}_1(t)$ and  $\Tilde{\Pi}_0$ in \eqref{tilde}. 
\\

\noindent \textbf{2.} From \eqref{pi0},  for  any $ x_0 \in \mathcal{X}_0$,  
\begin{align} A_0^* \Pi_0 x_0  &=   -   Q_0 x_0.
\end{align}  
Writing $P^{\mathcal{X}_0} x =x_0$, we use \eqref{relation2} and \eqref{operators2} to write the previous equation as 
\begin{align}
  A^* P^{\mathcal{Z}_0*}  \Pi_0 P^{\mathcal{X}_0}  x + Q_0 P^{\mathcal{X}_0} x = 0.
\end{align} 
 Referring to \eqref{pi tilde 0}, we find that the operator $\Tilde{\Pi}_0$ solves equation \eqref{equation-pi0-no-projection}. Deducing that $\Tilde{\Pi}_0$  is the unique solution  of \eqref{equation-pi0-no-projection} on $\mathcal X_0 $ is straightforward.
   If $x_i \in D(A)$, then   $(x_i)_1 \in D(A_1)$.  It follows from equation \eqref{solutin-x1-opt} that $x_1^{opt}(t) \in D(A)$, and equation \eqref{standard DRE} implies 
\begin{align}
    \frac{d}{dt}    \Pi_1(t) x^{opt}_1 (t) +  \Pi_1(t)   E_1^{-1} A_1 x_1^{opt}(t) +  A_1^* E_1^{-*}\Pi_1(t)  x_1^{opt}(t) & \nonumber \\
     -  \Pi_1(t)   \Tilde{B}_1   \Tilde{R}^{-1} \tilde{B}_1^* \Pi_1(t)  x_1^{opt}(t) +     Q_1 x_1^{opt}(t)  &=0 . \label{equation4.26} 
     \end{align}
    Since the right-hand-sides of equations \eqref{optimal-control-x1-pi1} and \eqref{optimal-control-pi1e-pi0} are  equal, 
     $$  -\Tilde{R}^{-1} \tilde{B}_1^* \Pi_1(t)  x_1^{opt}(t) =- R^{-1} \Big(\tilde{B}_1^* \Pi_1(t)   x_1^{opt}(t) + B^*_0  \Pi_0   x_0^{opt}(t) \Big),  $$ 
     and so 
     \begin{align}
          -  \Pi_1(t)   \Tilde{B}_1   \Tilde{R}^{-1} \tilde{B}_1^* \Pi_1(t)  x_1^{opt}(t) = -  \Pi_1(t)   \Tilde{B}_1 R^{-1} \Big(\tilde{B}_1^* \Pi_1(t)   x_1^{opt}(t) + B^*_0  \Pi_0   x_0^{opt}(t) \Big). \label{equivalent expression} 
     \end{align}
     Using  \eqref{equivalent expression} to replace  the term  $-  \Pi_1(t)   \Tilde{B}_1   \Tilde{R}^{-1} \tilde{B}_1^* \Pi_1(t)  x_1^{opt}(t)$ in  equation \eqref{equation4.26},
     \begin{align}
       \frac{d}{dt}    \Pi_1(t) x_1^{opt}(t) +  \Pi_1(t)   E_1^{-1} A_1 x_1^{opt}(t) +  A_1^* E_1^{-*}\Pi_1(t)   x_1^{opt}(t)& \nonumber \\
     - \Pi_1(t)   \Tilde{B}_1  R^{-1} \Big(\tilde{B}_1^*  \Pi_1(t)  x_1^{opt}(t) + B^*_0  \Pi_0   x_0^{opt}(t) \Big)  +     Q_1x_1^{opt}(t) &=0  \label{DRE-1}
\end{align}

From \eqref{relation2}, recall that for any $x_1(t) \in D(A)$
\begin{align}
    A_1 x_1(t) = A P^{\mathcal{X}_1} x(t) = P^{\mathcal{Z}_1} A x(t).
\end{align}
Hence, using the statement above along with \eqref{part1-defintions}, we 
rewrite equation \eqref{DRE-1} as  
  \begin{align*}
& \frac{d}{dt}   E^* P^{\mathcal{Z}_1*}  E_1^{-*} \Pi_1(t) E_1^{-1} P^{\mathcal{Z}_1}E x^{opt}(t)       + E^* P^{\mathcal{Z}_1*} E_1^{-*} \Pi_1(t) E_1^{-1} P^{\mathcal{Z}_1}A   x^{opt}(t)  \nonumber \\
& +  A^* P^{\mathcal{Z}_1*}E_1^{-*} \Pi_1(t) E_1^{-1} P^{\mathcal{Z}_1} E x^{opt}(t)  \nonumber \\
& -   E^* P^{\mathcal{Z}_1*} E_1^{-*} \Pi_1(t) E_1^{-1}  P^{\mathcal{Z}_1} B R^{-1}  B^* P^{\mathcal{Z}_1*} E_1^{-*} \Pi_1(t) E_1^{-1}  P^{\mathcal{Z}_1} E x^{opt}(t)  \nonumber \\
& -  E^* P^{\mathcal{Z}_1*} E_1^{-*} \Pi_1(t) E_1^{-1} P^{\mathcal{Z}_1}  B R^{-1}     B^* P^{\mathcal{Z}_0*} \Pi_0  P^{\mathcal{X}_0}  x^{opt}(t) +           Q_1 P^{\mathcal{X}_1}  x^{opt}(t) = 0 .  
\end{align*}
Using the definitions of $\Tilde{\Pi}_1(t)$ and $\Tilde{\Pi}_0$ in  \eqref{tilde}, the previous equation yields
     \begin{align}
\Big( \frac{d}{dt}&  E^* \Tilde{\Pi}_1(t)E  +E^* \Tilde{\Pi}_1(t) A    + A^* \Tilde{\Pi}_1(t) E           -  E^* \Tilde{\Pi}_1(t) B R^{-1}  B^* \Tilde{\Pi}_1(t)E     \nonumber \\
&-   E^* \Tilde{\Pi}_1(t) B R^{-1}     B^* \tilde{\Pi}_0  +   Q_1 P^{\mathcal{X}_1}  \Big) x_{opt} (t) = 0    . \label{equation}
\end{align}
Since the initial condition $x_i$ was  arbitrary  in  $D(A)$, and  equation \eqref{equation} holds for all $t \geq 0 ,$   it follows that $\Tilde{\Pi}_1(t)$ solves \eqref{equation-pi1-no-projection}, where $\tilde{\Pi}_0$ solves \eqref{equation-pi0-no-projection}.

The final condition in \eqref{final-pi1-no-projection} is obtained  from \eqref{final-pi1} with the help of \eqref{relations}  and \eqref{assumption1}.\\

\noindent \textbf{3.}  To obtain the minimum cost \eqref{min-cost-no-projection},  we refer to statement \eqref{min-cost-E}. For any $x_0 \in \mathcal{X}$,
\begin{align}
    &  J(x_i,u^{opt} ; t_f)  =   \langle  E_1 (x_i)_1 ,   E_1^{-*} \Pi_1(0) E_1^{-1}  E_1 (x_i)_1\rangle _{\mathcal{X}}  . 
\end{align}
Using \eqref{pi tilde 1}, the previous equation yields
\begin{align}
    &  J(x_i,u^{opt} ; t_f)  =   \langle  E x_i ,   \tilde{\Pi}_1(0) E x_i\rangle _{\mathcal{Z}}  . 
\end{align}
\end{proof}

The natural question we now address is whether the differential equation \eqref{riccati-eqs-no-projection} has a unique solution. The following theorem shows that the solution to this equation is unique on the range of  $E$.

\begin{theorem} Recall the operator $\Tilde{\Pi}_0$ in \eqref{pi tilde 0} and  the differential equation \eqref{equation-pi1-no-projection}
 \begin{subequations} \label{differential-equation} 
\begin{align}
 \frac{d}{dt}  E^*  Z(t)  E x  + E^*  Z(t)   Ax +   A^* Z (t)  E x& - E^* Z(t)    B   R^{-1} B^* Z(t)   E x  \nonumber \\
   -E^* Z(t)    B   R^{-1} B^*  \Tilde{\Pi}_0x +     Q P^{\mathcal{X}_1} x &=0 , \qquad \forall x \in D(A) ,  \label{DE} \\
    E^* Z (t_f) E x&= G x   .  \label{final-condition-DE}
\end{align}  
\end{subequations} 
\begin{enumerate}
    \item Let    $\Tilde{\Pi}_1(t) \in C ([0,t_f]; \mathcal{L}(\mathcal{Z}))$ be as defined by \eqref{pi tilde 1} where $\langle  Ex^a,  \Tilde{\Pi}_1(t) Ex^b \rangle_{\mathcal{X}} $ is differentiable for all $x^a, \; x^b \in D(A)$ and $ t \in (0,t_f) $. Also,   let $$Z_2(t) \in C ([0,t_f];  \mathcal{L} (\mathcal{Z}_0, \mathcal{Z}_1)), \quad Z_4(t) \in C ([0,t_f]; \mathcal{L} (\mathcal{Z}_0)),$$ where $Z_4 (t) $ is arbitrary and $Z_2(t) $  solves 
    \begin{equation} Z_2(t) (A_0- B_0  R^{-1} B_0^* \tilde \Pi_0) =  Z_1(t) B_1 R^{-1} B_0^* \tilde \Pi_0 . \label{eq-Z2}
    \end{equation}
    The general solution to equation \eqref{differential-equation} is  
    \begin{align}
       Z(t) = \Tilde{\Pi}_1(t) +   P^{\mathcal{Z}_1*} Z_2 (t)  P^{\mathcal{Z}_0}  + P^{\mathcal{Z}_0*} Z_4 (t) P^{\mathcal{Z}_0}  . 
    \end{align}

      \item For any solution  $  Z(t)$ of the equation \eqref{differential-equation}, the operator-valued function $ Z(t) E \in \mathcal{L} (\mathcal{X}, \mathcal{Z}) $ is unique and  leads to the optimal control \eqref{optimal-control-riccati-no-projection}  
      \begin{align} 
           u^{opt}(t) 
   =&-R^{-1} B^* (\Tilde{\Pi}_0 +       Z(t) E)  x^{opt}(t). \label{optimal-control-general-X}
      \end{align}
      The minimum cost is
      \begin{align}
          J(x_i,u^{opt} ; t_f)  =   \langle  Ex_i ,    Z(0) E x_i\rangle _{\mathcal{Z}} \label{min-cost-any X}
      \end{align}
      \end{enumerate}  
\end{theorem}

\begin{proof}
Decompose equation \eqref{DE} with the projections $\tilde{P}^{\mathcal{Z}}$ and $ \tilde{P}^{\mathcal{X}}$ in \eqref{P&D} and write an arbitrary solution $Z(t)$
$$Z(t) = \begin{bmatrix}
    Z_1 (t) & Z_2 (t) \\ Z_3 (t) & Z_4 (t) 
\end{bmatrix}$$ where 
\begin{align*}
&     Z_1(t) \in \mathcal{L}(\mathcal{Z}_1) , &\quad   Z_2(t) \in \mathcal{L}(\mathcal{Z}_1, \mathcal{Z}_0), \\
&     Z_3(t) \in \mathcal{L}(\mathcal{Z}_0,\mathcal{Z}_1 ) , &\quad   Z_4(t) \in \mathcal{L}(\mathcal{Z}_0). 
\end{align*}
Use the expression of $\Tilde{\Pi}_0$ in \eqref{pi tilde 0} and that $Q$ satisfies \eqref{assumption2} , to obtain 
\begin{align}
    &\frac{d}{dt}\begin{bmatrix}
         E_1^* & 0\\ 0 & 0 
     \end{bmatrix}\begin{bmatrix}
        Z_1(t) & Z_2(t) \\
          Z_3(t) & Z_4(t) \\
    \end{bmatrix}  \begin{bmatrix}
         E_1 & 0\\ 0 & 0 
     \end{bmatrix}  + \begin{bmatrix}
         E_1^* & 0\\ 0 & 0 
     \end{bmatrix}  \begin{bmatrix}
        Z_1(t) & Z_2(t) \\
          Z_3(t) & Z_4(t) \\
    \end{bmatrix}  \begin{bmatrix}
         A_1 & 0\\ 0 & A_0 
     \end{bmatrix}  \nonumber \\ 
    & +  \begin{bmatrix}
         A_1^* & 0\\ 0 & A^*_0 
     \end{bmatrix} \begin{bmatrix}
        Z_1(t) & Z_2(t) \\
          Z_3(t) & Z_4(t) \\
    \end{bmatrix}\begin{bmatrix}
         E_1 & 0\\ 0 & 0 
     \end{bmatrix} \nonumber \\
    & - \begin{bmatrix}
         E_1^* & 0\\ 0 & 0 
     \end{bmatrix} \begin{bmatrix}
        Z_1(t) & Z_2(t) \\
          Z_3(t) & Z_4(t) \\
    \end{bmatrix}   \begin{bmatrix}
        B_1 \\ B_0
    \end{bmatrix}   R^{-1}   \begin{bmatrix}
        B_1^* & B_0^*
    \end{bmatrix} \begin{bmatrix}
        Z_1(t) & Z_2(t) \\
          Z_3(t) & Z_4(t) \\
    \end{bmatrix} \begin{bmatrix}
         E_1 & 0\\ 0 & 0 
     \end{bmatrix}  \nonumber \\
 &  -\begin{bmatrix}
         E_1^* & 0\\ 0 & 0 
     \end{bmatrix}  \begin{bmatrix}
        Z_1(t) & Z_2(t) \\
          Z_3(t) & Z_4(t) \\
    \end{bmatrix} \begin{bmatrix}
        B_1 \\ B_0
    \end{bmatrix}   R^{-1}    \begin{bmatrix}
   0 & B_0^*  \Pi_0  P^{\mathcal{Z}_0}
    \end{bmatrix}       +   \begin{bmatrix}
         Q P^{\mathcal{X}_1} & 0\\ 0 & 0 
     \end{bmatrix}   = \begin{bmatrix}
        0& 0 \\ 0& 0
    \end{bmatrix} .\nonumber \\
    &\label{long equation} 
\end{align}
Similarly,  use the assumption on $G$ in \eqref{assumption1} to  decompose the final condition \eqref{final-condition-DE} as
\begin{align}
   \begin{bmatrix}
         G_1 & 0\\ 0 & 0 
     \end{bmatrix}  &=  \begin{bmatrix}
         E_1^* & 0\\ 0 & 0 
     \end{bmatrix} \begin{bmatrix}
        Z_1(t_f) & Z_2(t_f) \\
          Z_3(t_f) & Z_4(t_f) \\
    \end{bmatrix} \begin{bmatrix}
         E_1 & 0\\ 0 & 0 
     \end{bmatrix} \\
   &=   \begin{bmatrix}
      E_1^*  Z_1(t_f) E_1 & 0 \\
          0 & 0 \\
    \end{bmatrix} .
\end{align}
The resulting 4 equations yield  $Z_3 (t)\equiv0$, $Z_4(t) $  is arbitrary, $Z_2 (t)$ solves \eqref{eq-Z2} and,
recalling that  $\Pi_1(t) $ is the unique solution of \eqref{standard DRE},
\begin{align}
   Z_1(t)=   E_1^{-*} \Pi_1(t) E_1^{-1} . 
\end{align}

\noindent \textbf{2.} This assertion can be shown by substituting for the general solution of  \eqref{riccati-eqs-no-projection} in equation \eqref{optimal-control-general-X}. Using that $P^{\mathcal{Z}_0} E=0$, we arrive at the expression of the optimal control \eqref{optimal-control-riccati-no-projection}. Similarly, we can find the minimum cost from \eqref{min-cost-any X}.
\end{proof}

 


Equations \eqref{composition} and the projections $P^{\cX}$, $P^{\cZ}$ were used to derive \eqref{riccati-eqs-no-projection}. However, the restrictions of $Q$ on the subspaces  $\mathcal{X}_1$ and $\mathcal{X}_0$ are found, solving the system \eqref{riccati-eqs-no-projection} requires no knowledge of the projections.   

When comparing system \eqref{riccati-eqs-no-projection} with the optimality equations derived in finite-dimensional space, it is apparent that system \eqref{riccati-eqs-no-projection} differs from the ones  derived in \cite{bender1987linear, ACC} since the knowledge of the restrictions such as $E_1$, $A_1$, etc is not required here. Assuming no penalization on the $L_2$-norm of the algebraic state, i.e. $Q_0 P^{\mathcal{X}_0} = 0$, system \eqref{riccati-eqs-no-projection} reduces to a single differential Riccati equation that mirrors the one derived for  finite-dimensional DAEs through the use of  behaviors \cite{reis2019linear}.

\section{Numerical simulations}\label{section5}
In order to illustrate the theoretical results, we study the following class of coupled systems 
\begin{subequations} \label{coupled system}
    \begin{align}
    w_t(x,t) =&  w_{xx}(x,t) - \rho w(x,t) + \alpha v (x,t) + u(t) , \label{p1}\\
   0 =&  v_{xx}(x,t) - \gamma v(x,t) + \beta w(x,t) +u(t), \label{p2}
\end{align}
 with the  boundary conditions
\begin{align}
   w_x(0,t)=& 0, \quad w _x(1,t) =0, \label{p3}\\
    v_x (0,t)=&0, \quad v_x (1,t) =0,  \label{p4}
\end{align}
\end{subequations}
where $x \in [0,1]$ and $t \geq 0$. Also, $w(x,0)=\sin(\pi x)$ and $v(x,0)= \beta (\gamma - d_{xx})^{-1} \sin(\pi x).$ 
The system's parameters are \(\rho=1\), \(\gamma=\alpha =\beta=2\) . For all $n$, $\gamma \neq -(n \pi)^2$, it then $\gamma$ is not an eigenvalue of the operator $\partial_{xx}$ \cite[Example 8.1.8]{curtain2020introduction}. Consequently, defining $\mathrm{Z} =H_2(0,1)$ and $I$ to be  the identity operator on $\mathrm{Z}$, the operator $ \gamma I - \partial_{xx}$ is invertible. It will  also prove useful to define 
\begin{equation*}
    A_d w = \frac{d^2}{dx^2} w, \; \; D(A_d)=\{w \in  \mathrm{Z}: w_x(0)=w_x(1)=0 \},
\end{equation*}
and 
\begin{align*}
    x(t) = \begin{bmatrix}
        w(x,t) \\ v(x,t)
    \end{bmatrix} \in \mathcal{X}=\mathrm{Z}\times \mathrm{Z}.
\end{align*}
System \eqref{coupled system} can now be written in the form  \eqref{1}
    \begin{align*}
    \frac{d}{dt} \underbrace{\begin{bmatrix}
        I  & 0 \\ 0 &0 \end{bmatrix}}_{E} x(t)&= \underbrace{\begin{bmatrix}
        A_d- \rho I  && \alpha I \\   \beta I &&   A_d- \gamma I 
    \end{bmatrix} }_{A} x(t)+  \underbrace{\begin{bmatrix}
        I  \\ I
    \end{bmatrix}}_{B} u(t),   \nonumber \\
    x(0) &= \begin{bmatrix}
        \sin(\pi x) & \beta (\gamma - d_{xx})^{-1} \sin(\pi x)
    \end{bmatrix}^*
\end{align*}
 System \eqref{coupled system}  was shown in \cite[section 4 ]{jacob2022solvability} to be  radial of degree 0.  
 To approximate this coupled system, use finite-element method with linear splines to obtain a system of DAEs. The spatial interval $[0,1]$ is subdivided into 27 equal intervals. The dynamics of the parabolic and elliptic states without control (i.e. $u(t) \equiv 0$) is given in Figure \ref{LQ} a \& b.

 Define the  cost functional
\begin{align}
J(x_i,u ; 6) =&   \int_{0}^{6} \langle x(s),  \begin{bmatrix}
    I&0\\0&0
\end{bmatrix}  x(s)\rangle_{\mathrm{Z}}  +  \langle u(s),  u(s)\rangle_{\mathcal{U}} ds . \label{cost-numerical}
\end{align}
Comparing the previous cost with \eqref{cost}, it is clear that $t_f=6$, $Q=   \begin{bmatrix}
    I&0\\0&0
\end{bmatrix}$ , $G= \begin{bmatrix}
    0&0\\0&0
\end{bmatrix}$.  Note that projections $P^{\cX}$ and $P^{\cZ}$, which were calculated in \cite[section 4 ]{jacob2022solvability}, can be used used  to obtain
\begin{align}
   Q P^{\mathcal{X}_1}= \begin{bmatrix}
    I&0\\0&0
\end{bmatrix}, \quad    Q P^{\mathcal{X}_0}= \begin{bmatrix}
    0&0\\0&0
\end{bmatrix}  .
\end{align}
To demonstrate the finding in the previous sections, we solve system \eqref{riccati-eqs-no-projection} for the optimal control \eqref{optimal-control-riccati-no-projection}. Since this system consists of  operator equations on an infinite-dimensional Hilbert space, it cannot be solved exactly. Therefore, the control is calculated using an finite-dimensional approximation. 
The convergence of the approximation method  to  the true optimal one and closed-loop performance can be discussed by doing  calculations of different approximation order and following a similar approach as in \cite[Chapter 4]{morris2020controller}. This task will be addressed in the future work. 
Without the need of decomposing the state $x(t)$ or calculation operators $E_1$, $A_1$ etc, we now solve the finite-dimensional approximation of system \eqref{riccati-eqs-no-projection}. This is done via using ``ode15s" which is based on a backward differential formula (BDF).  Consequently, we obtain an approximation of  the optimal control \eqref{optimal-control-riccati-no-projection}. 
The approximated control signal   at $x=0$ \eqref{optimal-control-riccati-no-projection}  is given  in Figure \ref{optimal-control-DRE}. Note that control $u(t)$ ensures the consistency statement on the initial conditions, which is  
\begin{align*}
    u(x,0)&= \frac{(d_{xx} - \gamma I)}{\beta} v(x,0) - \beta w(x,0) \\
    &= 0.
\end{align*}
Figure \eqref{LQ} c \& d illustrates the dynamics of the coupled system  after applying  control \eqref{optimal-control-riccati-no-projection}.

\begin{figure}[H]
    \centering
    \includegraphics[scale=0.5]{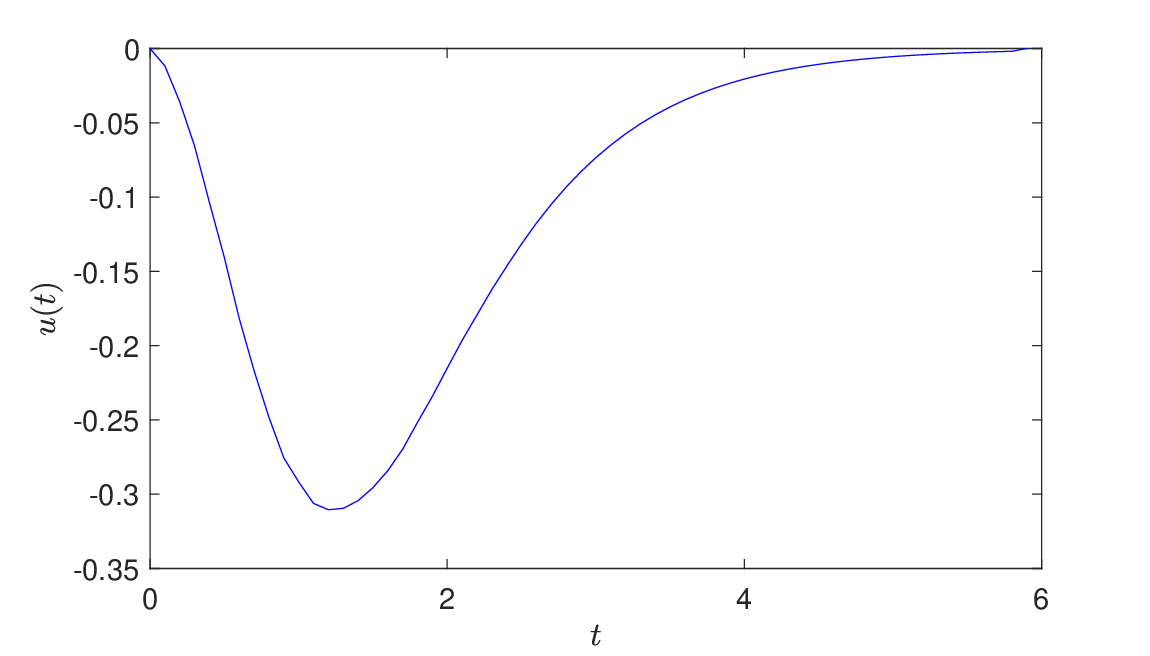}
    \caption{LQ-optimal feedback control at $x=0$ . This control signal minimizes the cost functional  \eqref{cost-numerical}, and is derived by solving system \eqref{riccati-eqs-no-projection} after discretization.   }
    \label{optimal-control-DRE}
\end{figure}
\begin{figure*}[htb]
\centering
\begin{minipage}[b]{.45\linewidth}
  \centering
  \includegraphics[width=\linewidth]{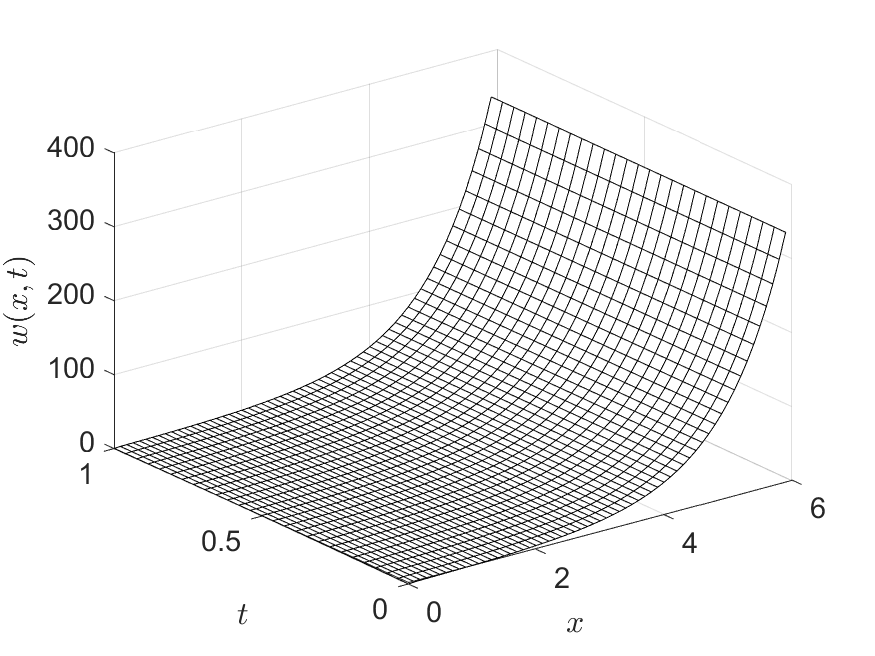}
  \\(a) Uncontrolled $w(x,t)$
\end{minipage}%
\hfill 
\begin{minipage}[b]{.45\linewidth}
  \centering
  \includegraphics[width=\linewidth]{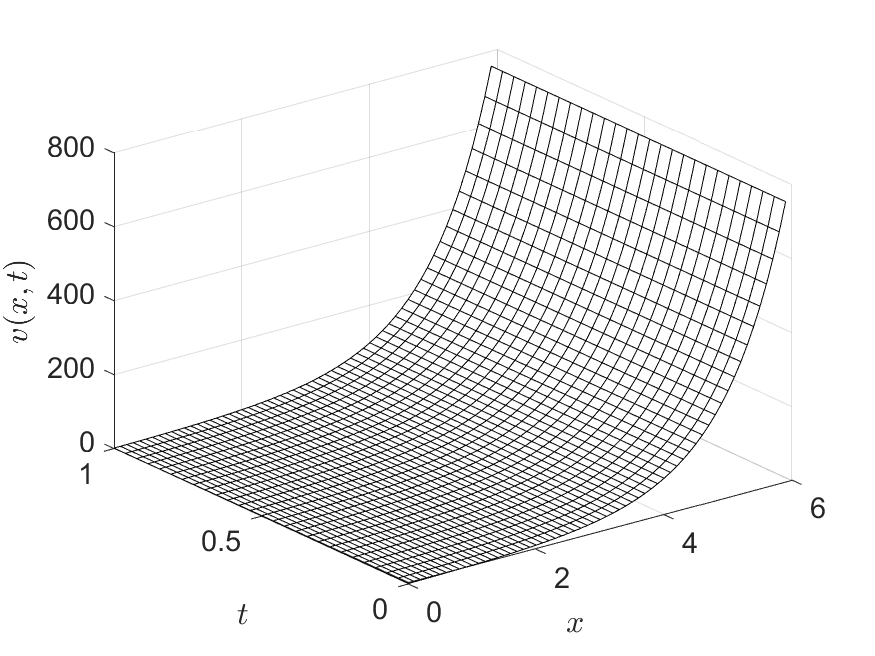}
  \\(b) Uncontrolled $v(x,t)$
\end{minipage}
\newline\newline 
\begin{minipage}[b]{.45\linewidth}
  \centering
  \includegraphics[width=\linewidth]{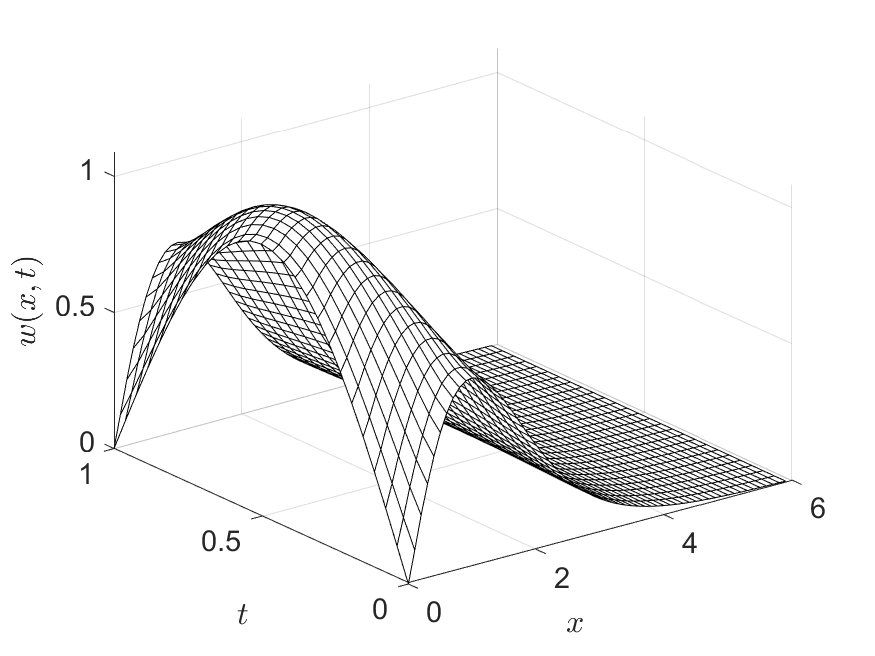}
  \\(c) Controlled $w(x,t)$
\end{minipage}
\hfill 
\begin{minipage}[b]{.45\linewidth}
  \centering
  \includegraphics[width=\linewidth]{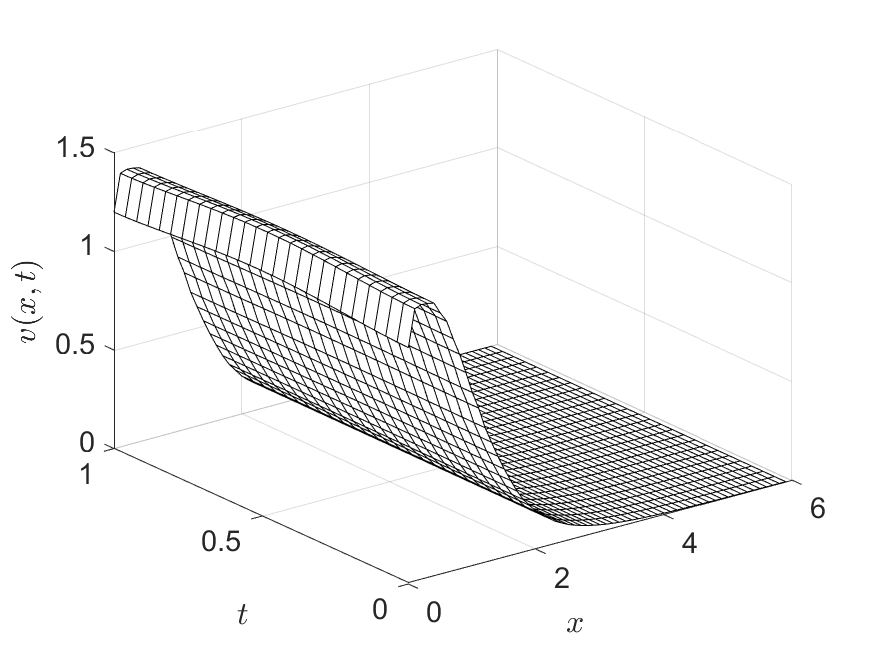}
  \\(d) Controlled $v(x,t)$
\end{minipage}
\caption{A 3D landscape of the dynamics of a coupled parabolic-elliptic system with initial condition \(w_0=\sin(\pi x)\), \(v_0= \beta (d_{xx} - \gamma I)^{-1} \sin(\pi x)\), without and with control \eqref{optimal-control-riccati-no-projection}. The parameters of the system are \(\rho=1\), \(\gamma=\alpha =\beta=2.\) The uncontrolled system is unstable, but the use of LQ feedback control causes the states to decay towards zero.}
\label{LQ}
\end{figure*}

\section{Conclusions}

This paper  extends the classical finite-time linear quadratic control problem for finite-dimensional DAEs into the  infinite-dimensional case. We showed the existence of  a continuous optimal control that ensures the consistency of the initial conditions while  minimizing the cost functional. Decomposing  the PDAE into a  Weierstra$\beta$  canonical form was crucial in the proofs. 
However, the optimal control can be calculated  derived differential Riccati equation without projecting of any operators other than the state weight.
Future work will consider the minimization problem on an infinite-time horizon. 

The class of PDAEs  considered here was restricted to those with radiality-index  zero. We aim to extend the results to higher-index PDAEs in future research. This already been done \cite{ACC} for finite-dimensional DAEs without the use of behaviors.
Another prospective research problem is to examine the linear quadratic control problem for linear PDAEs in situations where the operator associated with the penalization on the $L_2$-norm of the control input is not necessarily invertible, since as illustrated in \cite{reis2019linear} a unique solution can exist for the optimization problem in the finite-dimensional situation.  Finally, since  the derived differential Riccati equation is an  operator equation on an infinite-dimensional Hilbert space, numerical treatment for solving needs further attention. 

\bibliographystyle{siam}
\bibliography{ex_shared}
\end{document}

%% file: ex_shared.tex
\usepackage{lipsum}
\usepackage{comment}
\usepackage{algorithmic}
 \usepackage{color}

 \usepackage[normalem]{ulem}

 \ifpdf
 \DeclareGraphicsExtensions{.eps,.pdf,.png,.jpg}
\else
 \DeclareGraphicsExtensions{.eps}
 \fi

\newsiamremark{remark}{Remark}
\newsiamremark{defn}{Definition}
\newsiamremark{lem}{Lemma}
\newsiamremark{cor}{Corollary}
\newsiamremark{prop}{Proposition}
\newsiamremark{exmp}{Example}
\newsiamremark{hypothesis}{Hypothesis}
\crefname{hypothesis}{Hypothesis}{Hypotheses}
\newsiamthm{claim}{Claim}
\headers{FINITE-TIME LINEAR-QUADRATIC OPTIMAL CONTROL OF INDEX-1 PDAEs}{A. Alalabi and K. Morris}

\title{Finite-time Linear-Quadratic optimal control of partial differential-algebraic equations \thanks{This work was supported by NSERC (Canada) and the Faculty of Mathematics, University of Waterloo.}}

\author{Ala' Alalabi \thanks{Authors are with the  Department of Applied Mathematics,
        University of Waterloo, 200 University Avenue West, Waterloo, ON, Canada  N2L 3G1  \\
  \email{aalalabi@uwaterloo.ca},  \email{ kmorris@uwaterloo.ca} .}
\and Kirsten Morris}

\usepackage{amsopn}

\makeatletter
\newcommand*{\addFileDependency}[1]{
  \typeout{(#1)}
  \@addtofilelist{#1}
  \IfFileExists{#1}{}{\typeout{No file #1.}}
}
\makeatother
